\documentclass[12pt, reqno]{amsart}
\usepackage{amsmath,times,epsfig,amssymb,amsbsy,amscd,amsfonts,amstext,color,bm}

\usepackage{bm}
\usepackage{comment}

\usepackage{tikz}

\numberwithin{equation}{section}

\theoremstyle{plain}
\newtheorem{theorem}{Theorem}[section]

\newtheorem{lemma}[theorem]{Lemma}

\newtheorem{proposition}[theorem]{Proposition}

\newtheorem{question}[theorem]{Question}

\theoremstyle{definition}
\newtheorem{definition}[theorem]{Definition}
\newtheorem{example}[theorem]{Example}

\theoremstyle{remark}
\newtheorem{remark}[theorem]{Remark}

\newcommand{\A}{\mathcal{A}}

\newcommand{\Z}{\mathbb{Z}}
\newcommand{\K}{\mathbb{K}}

\newcommand{\R}{\mathbb{R}}
\newcommand{\C}{\mathbb{C}}

\newcommand{\scL}{\mathcal{L}}

\newcommand{\eps}{\varepsilon}
\newcommand{\vf}{\theta}

\newcommand{\BM}{\operatorname{BM}}

\newcommand{\codim}{\operatorname{codim}}

\begin{document}

\title[Embedded sphere]{A construction of homotopically non-trivial 
embedded spheres for hyperplane arrangements}

\begin{abstract}
We introduce the notion of locally consistent system of 
half-spaces for a real hyperplane arrangement. 
By shifting the real unit sphere into the imaginary direction 
indicated by the half-spaces, we embed a sphere in the complexified 
complement. 
We then prove that the sphere is homotopically trivial 
if and only if the system 
of half-spaces is globally consistent. 
To prove its non-triviality, we compute the twisted intersection number 
of the sphere with a specific, explicitly constructed twisted 
Borel-Moore cycle. 
\end{abstract}

\author{Masahiko Yoshinaga}
\address{Masahiko Yoshinaga, 
%Department of Mathematics, Faculty of Science, Hokkaido University, 
%Kita 10, Nishi 8, Kita-Ku, Sapporo 060-0810, Japan.
Osaka University}
\email{yoshinaga@math.sci.osaka-u.ac.jp}

%\thanks{This work was partially supported by 
%JSPS KAKENHI Grant Numbers JP23H00081}

\subjclass[2010]{Primary 14N20, Secondary 52C35}
%Primary 14C21,
%14F99, 32S22 ; Secondary 14E05, 14H50.}

%52C35(2000–now)Arrangements of points, flats, hyperplanes
%20F55(1973–now)Reflection and Coxeter groups
%14N20(2000–now)Configurations and arrangements of linear subspaces
%32S22(2000–now)Relations with arrangements of hyperplanes
%14C21(1980–now)Pencils, nets, webs in algebraic geometry 
%14(1940–now)Algebraic geometry
%14F(1973–now)(Co)homology theory in algebraic geometry [See also 13Dxx]
%14F99(1973–now)None of the above, but in this section
%14E05(1973–now)Rational and birational maps
%14H50(1980–now)Plane and space curves

\keywords{Hyperplane arrangements, homotopy groups, 
$K(\pi, 1)$ space, local system homology}

\dedicatory{Dedicated to 
Prof. Enrique Artal Bartolo on the occasion of his 60th birthday}

\date{\today}

\maketitle

\tableofcontents

\section{Introduction}
\label{sec:intro}

A complex hyperplane arrangement $\A$ in $\C^\ell$ 
is called $K(\pi, 1)$ if the complement 
$M=\C^\ell\setminus\bigcup_{H\in\A}H$ 
is a $K(\pi, 1)$ space, that is, all higher homotopy 
groups vanish. Research on $K(\pi, 1)$-arrangements 
dates back to 1960's \cite{fox-neu}. 
In 1971, Brieskorn \cite{bri} raised a problem regarding the $K(\pi, 1)$ property 
of Coxeter and reflection arrangements. This question has inspired 
numerous significant works \cite{bessis, cha-dav, del-kpi1, fal-wt, nak, pao-sal}. 
For surveys on this topic, see \cite{fr-hom1, fr-hom2, par-conj}. 

However, not all hyperplane arrangements are $K(\pi, 1)$. Actually, 
Hattori \cite{hat-top} proved that 
generic hyperplane arrangements have non-vanishing 
higher homotopy groups, and hence 
are not $K(\pi, 1)$. Subsequently, 
several new non-$K(\pi, 1)$ arrangements have been discovered 
and studied \cite{ps-h, yos-twi}. 

However, we are still far from characterizing $K(\pi, 1)$ 
arrangements. There is no effective method to 
determine whether a hyperplane arrangement $\A$ is 
$K(\pi, 1)$. This problem remains challenging even for affine 
line arrangements in $\R^2$ 
(or equivalently, central arrangements in $\R^3$). 

One of the reasons for the difficulty in studying the homotopy 
groups of the complement $M$ is the vanishing of the 
Hurewicz map. Randell \cite{ran-hom} proved 
that the Hurewicz map 
\[
\pi_k(M)\longrightarrow H_k(M, \Z)
\]
always vanishes for $k\geq 2$. This means that even if 
the homotopy group $\pi_k(M)$ is not vanishing, it can not be 
detected by the homology group. 

%%%%%

In this paper, we develop a method to construct an embedded 
sphere in the complement and detecting non-triviality in the 
homotopy group. The paper is organized as follows. 

In \S \ref{sec:sys}, we introduce the systems of half-spaces 
and related notions. 
The notion of locally consistent 
system of half-spaces will play a key role in constructing 
an embedded sphere. 

In \S \ref{sec:sphere}, we first 
recall a useful way to describe the complex 
vector space $\C^\ell$ in terms of tangent vectors in $\R^\ell$, 
a method used in \cite{yos-lef} to study hyperplane arrangements. 
Using this description, for a given locally consistent system of 
half-spaces, we construct an embedded sphere in the complexified 
complement. The idea involves shifting real sphere to the imaginary 
direction specified by the system of half-spaces. The main 
result in this paper characterizes the system of half-spaces 
for which the resulting sphere is homotopically non-trivial. 

In \S \ref{sec:proof}, we prove the main result. The idea of the proof is 
to construct a twisted (local system coefficients) Borel-Moore 
homology cycle that has a non-zero intersection number with 
the embedded sphere.

In the final section \S \ref{sec:cor}, we formulate an 
obstruction for an arrangement to be $K(\pi, 1)$. 
We also provide several examples.

\section{Locally consistent systems of half-spaces}

\label{sec:sys}

\subsection{Signs and system of half-spaces}

\label{sec:signs}

Let $\A=\{H_1, \dots, H_n\}$ be a hyperplane arrangement 
in $V=\R^\ell$. 
We assume $\A$ is central and essential, meaning that 
$\bigcap_{i=1}^n H_i=\{0\}$. Fix a defining linear form 
$\alpha_i\in V^*$ to each hyperplane $H_i\in\A$. 
Define the half-space $H_i^+$ (resp. $H_i^-$) as 
$\{x\in V\mid \alpha_i(x)>0\ (\mbox{ resp. }\alpha_i(x)<0)\}$. 
The sign vector 
$\bm{\eps}=(\eps_1, \dots, \eps_n)\in\{\pm\}^n$ determines a 
system of half-spaces 
$h(\bm{\eps}):=(H_1^{\eps_1}, \dots, H_n^{\eps_n})$. 

\begin{definition}
The system of half-spaces $h(\bm{\eps})$ is \emph{consistent} 
(or \emph{globally consistent}) 
if $\bigcap_{i=1}^n H_i^{\eps_i}\neq\emptyset$. 
If not, we call $h(\bm{\eps})$ an inconsistent system. 
(See Figure \ref{fig:consistent}. The positive half-space is 
indicated by an arrow.) 
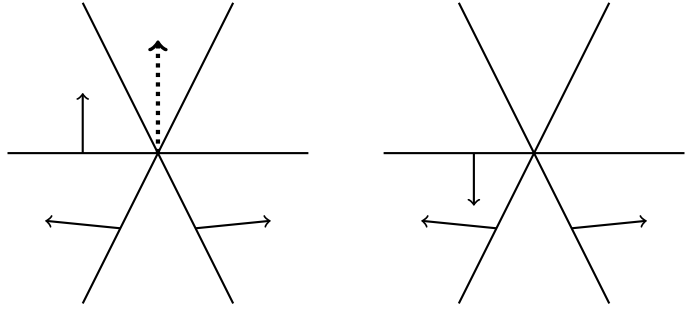
\begin{figure}[htbp]
\centering
\begin{tikzpicture}%[scale=0.8]

%%格子
%\draw [help lines] (0,0) grid (12,6);%(0,0)から(10,4)までの"細線の方眼"

\draw[thick] (1,0)--(3,4);
\draw[thick] (1,4)--(3,0);
\draw[thick] (0,2)--(4,2);

\draw[thick, ->] (1.5,1)--(0.5,1.1);
\draw[thick, ->] (2.5,1)--(3.5,1.1);
\draw[thick, ->] (1,2)--(1,2.8);
\draw[ultra thick, dotted, ->] (2,2)--(2,3.5);

\draw[thick] (6,0)--(8,4);
\draw[thick] (6,4)--(8,0);
\draw[thick] (5,2)--(9,2);

\draw[thick, ->] (6.5,1)--(5.5,1.1);
\draw[thick, ->] (7.5,1)--(8.5,1.1);
\draw[thick, ->] (6.2,2)--(6.2,1.3);

\end{tikzpicture}
\caption{Consistent (left) and inconsistent (right) system of 
half-spaces} 
\label{fig:consistent}
\end{figure}
\end{definition}
When $h(\bm{\eps})$ is consistent, the intersection 
$\bigcap_{i=1}^n H_i^{\eps_i}$ is clearly a chamber (a connected 
component of the complement $V\setminus\bigcup_{i=1}^nH_i$). 
Thus, a consistent system of half-spaces determines a 
chamber. 

Let $v\in V$ be a vector which is not contained in the union 
of hyperplanes $\bigcup_{i=1}^n H_i$. Then, one can associate 
a consistent system of half-spaces $h(\bm{\eps})$ such that 
$v\in H_i^{\eps_i}$ for each $i=1, \dots, n$. 
\begin{definition}
Let $h(\bm{\eps})$ be a system of half-spaces. A chamber $C$ 
is called a \emph{sink} if $C\subset H_i^{\eps_i}$ for each 
wall $H_i$ of $C$. 
\end{definition}
If $h(\bm{\eps})$ is consistent, then 
the intersection $\bigcap_{i=1}^n H_i^{\eps_i}$ is the unique sink. 
In general, one can prove the following. 
\begin{proposition}
\label{prop:sink}
Let $h(\bm{\eps})$ be a system of half-spaces. Then 
there exists a sink $C$. 
\end{proposition}
\begin{proof}
Let us call a sequence of chambers $(C_0, C_1, \dots, C_k)$ 
a \emph{flow} with respect to $h(\bm{\eps})$ 
if $C_{p-1}$ and $C_p$ ($1\leq p\leq k$) are adjacent, separated 
by the unique hyperplane $H_{i_p}$, and $C_p\subset H_{i_p}^{\eps_{i_p}}$ 
(Figure \ref{fig:flow}). Any flow can cross a hyperplane 
$H\in\A$ at most once. Hence, any flow has length at most 
$n$, which means that any flow terminates at a sink. 
\begin{figure}[htbp]
\centering
\begin{tikzpicture}%[scale=0.8]

%%格子
%\draw [help lines] (0,0) grid (12,6);%(0,0)から(10,4)までの"細線の方眼"

\draw[thick] (1,-0.2)--(1,4);
\draw[thick] (0.9,-0.2)--(3,4);
\draw[thick] (0.5,1)--(5,1);
\draw[thick] (0.5,2.5)--(5,2.5);
\draw[thick] (3.5,0)--(1.5,4);
\draw[thick] (2.4,-0.2)--(4.5,4);
\draw[thick] (2,4)--(5,3);
\draw[thick] (2,0.5)--(5,2);

\draw[thick, ->] (1,2)--++(0.3,0) node[below]{\scriptsize $C_0$};
\draw[thick, ->] (1.5,2.5)--++(0,0.3);
\draw[thick, ->] (2.7,2.5)--++(0,0.3);
\draw[thick, ->] (2.7,2.5)--++(0,0.3);
\draw[thick, ->] (1.2,1)--++(0,0.3);
\draw[thick, ->] (2,1)--++(0,0.3) node[right]{\scriptsize $C_1$};
\draw[thick, ->] (2,2)--++(0.2,-0.2);
\draw[thick, ->] (2.5,2)--++(0.2,0.1) node[right]{\scriptsize $C_2$};
\draw[thick, ->] (3.5,2)--++(-0.2,0.1);
\draw[thick, ->] (2.5,3)--++(0.3,0) node[right]{\scriptsize $C_3$};
\draw[thick, ->] (3.5,3.5)--++(0,-0.3);
\draw[thick, ->] (4,3)--++(-0.3,0);
\draw[thick, ->] (4,1.5)--++(0.3,-0.1);

%\draw[thick, ->] (6.2,2)--(6.2,1.3);

\end{tikzpicture}
\caption{The flow $(C_0, C_1, C_2, C_3)$ terminates at a sink.} 
\label{fig:flow}
\end{figure}
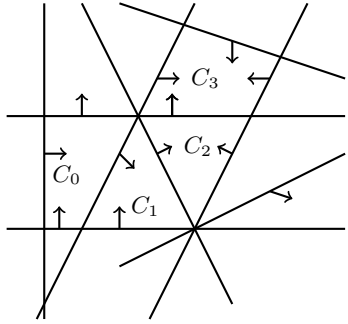
\end{proof}

\subsection{Locally consistent system of half-spaces}
\label{sec:loc}

Let 
$L(\A)=\left\{X=\bigcap_{H\in\mathcal{B}}H
\mid\mathcal{B}\subset\A\right\}$ be the set of intersections 
of a central hyperplane arrangement $\A=\{H_1, \dots, H_n\}$. 
\begin{definition}
\label{def:atX}
Let $X\in L(\A)$. 
Recall that 
$\A_X=\{H\in\A\mid H\supset X\}$ is the localization of $\A$ at $X$. 
We say a system of half-spaces $h(\bm{\eps})$ 
is \emph{consistent at $X$} if the subsystem 
$h(\bm{\eps}_X):=(H_i^{\eps_i}\mid H_i\in\A_X)$ is consistent. 
\end{definition}

\begin{remark}
Note that $h(\bm{\eps})$ is consistent at any hyperplane $H\in\A$. 
More generally, if $\A_X$ is Boolean, meaning $\A_X$ consists of 
independent hyperplanes, then any system of half-spaces 
$h(\bm{\eps})$ is consistent at $X$. 
\end{remark}

\begin{definition}
\label{def:locallycon}
Let $\A=\{H_1, \dots, H_n\}$ be a central essential arrangement. 
We say a system of hyperplanes $h(\bm{\eps})$ is 
\emph{locally consistent} if it is consistent at $X$ for 
any $X\in L(\A)\setminus\{0\}$. 
\end{definition}

\begin{example}
Consider the arrangement  $\A=\{H_1, H_2, H_3, H_4\}$ defined by 
$H_1=\{x=0\}, H_2=\{y=0\}, H_3=\{z=0\}, H_4=\{x+y+z=0\}$. 
Consider the system of half-spaces defined by the 
sign vector $\bm{\eps}=(+, +, +, -)$. 
Clearly, the intersection of half-spaces 
$x>0, y>0, z>0$, and $x+y+z<0$ is empty. Hence it is globally inconsistent. 

However, since any codimension $2$ intersection of $\A$ is 
Boolean, it is locally consistent (Figure \ref{fig:generic4}). 
The chamber $C$ determined by inequalities $x>0, y>0, z>0$ 
is a sink. 
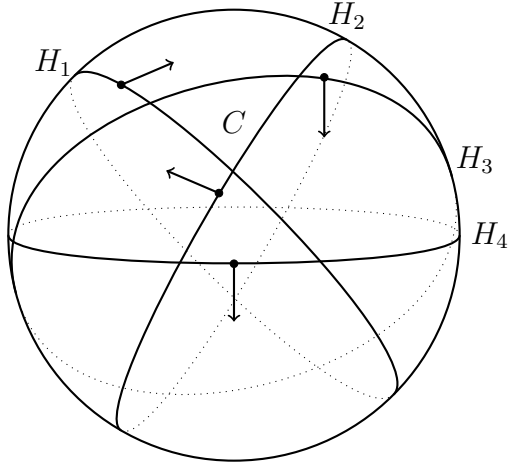
\begin{figure}[htbp]
\centering
\begin{tikzpicture}%[scale=0.8]

%%格子
%\draw [help lines] (0,0) grid (12,6);%(0,0)から(10,4)までの"細線の方眼"
%\draw [help lines] (-3,0) grid (5,6);%(0,0)から(10,4)までの"細線の方眼"

\draw[thick] (0,0) circle (3);

%\draw [bend right = 60] (-3,0) to (3,0);
%\draw [dotted, bend left = 60] (-3,0) to (3,0);

%\draw [thick] (-3,0) .. controls (-3,-0.5) and (3,-0.5) ..(3,0);
%\draw [dotted, thick] (-3,0) .. controls (-3,0.5) and (3,0.5) ..(3,0);

\draw [thick] (-3,0) .. controls (-3,-0.5) and (3,-0.5) ..(3,0);
\draw [dotted, thin] (-3,0) .. controls (-3,0.5) and (3,0.5) ..(3,0);

\draw [dotted, thin, rotate=60] (-3,0) .. controls (-3,-0.6) and (3,-0.6) ..(3,0);
\draw [thick, rotate=60] (-3,0) .. controls (-3,0.6) and (3,0.6) ..(3,0);

\draw [dotted, thin, rotate=-45] (-3,0) .. controls (-3,-0.8) and (3,-0.8) ..(3,0);
\draw [thick, rotate=-45] (-3,0) .. controls (-3,0.8) and (3,0.8) ..(3,0);

\draw [dotted, thin, rotate=15] (-3,0) .. controls (-3,-2.7) and (3,-2.7) ..(3,0);
\draw [thick, rotate=15] (-3,0) .. controls (-3,2.7) and (3,2.7) ..(3,0);

\draw (3,0) node[right]{$H_4$}; 
\filldraw[fill=black, draw=black] (0,-0.38) circle (0.05);
\draw[thick, ->] (0,-0.35)--++(0,-0.8);

\draw (2.8,1) node[right]{$H_3$}; 
\filldraw[fill=black, draw=black] (1.2,2.1) circle (0.05);
\draw[thick, ->] (1.2,2.1)--++(0,-0.8);

\draw (1.5,2.6) node[above]{$H_2$}; 
\filldraw[fill=black, draw=black] (-0.2,0.56) circle (0.05);
\draw[thick, ->] (-0.2,0.56)--++(-0.7,0.3);

\draw (-2,2.3) node[left]{$H_1$}; 
\filldraw[fill=black, draw=black] (-1.5,2) circle (0.05);
\draw[thick, ->] (-1.5,2)--++(0.7,0.3);

\draw (0,1.5) node {$C$}; 

\end{tikzpicture}
\caption{A locally consistent system of half-spaces with a sink $C$.} 
\label{fig:generic4}
\end{figure}
\end{example}

\section{Embedded spheres}
\label{sec:sphere}

\subsection{Complexified complements}
\label{sec:cpx}

The complex vector space $\C^\ell=\R^\ell\otimes\C
=\R^\ell+\sqrt{-1}\cdot\R^\ell$ can be 
identified with the total space of the tangent bundle $T\R^\ell$ 
of the space $\R^\ell$ by the map: $v\in T_{x}\R^\ell
\longmapsto x+\sqrt{-1}\cdot v\in\C^\ell$. 

Let $H\subset V=\R^\ell$ be a hyperplane defined by a linear 
form $\alpha\in V^*$. Then the complex point 
$x+\sqrt{-1}\cdot v$ is contained in the 
complexified hyperplane $H_\C=H\otimes\C$ if and only if 
\[
\alpha(x+\sqrt{-1}\cdot v)=
\alpha(x)+\sqrt{-1}\cdot\alpha(v)=0, 
\]
which is equivalent to both the real part ($x$) and the imaginary part 
($v$) being contained in $H$. 

Let $\A=\{H_1, \dots, H_n\}$ be a central arrangement in 
$V=\R^\ell$. Using the above description, one can describe the 
complexified complement $M(\A)=\C^\ell\setminus
\bigcup_{i=1}^n H_\C$ as follows (see also Figure \ref{fig:compl}). 
\[
%M(\A)=
\{x+\sqrt{-1}\cdot v\mid \mbox{ if 
$x\in H_i$ for some $1\leq i\leq n$, then }
v\notin T_{x}H_i\}. 
\]
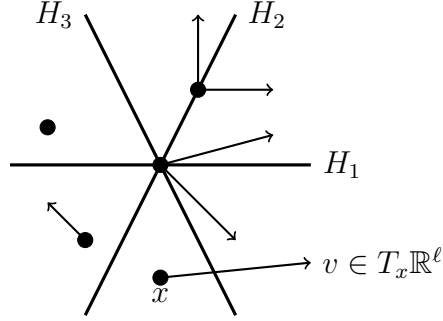
\begin{figure}[htbp]
\centering
\begin{tikzpicture}%[scale=0.8]

%%格子
%\draw [help lines] (0,0) grid (12,6);%(0,0)から(10,4)までの"細線の方眼"

\draw[very thick] (0,2)--(4,2) node[right] {$H_1$};
\draw[very thick] (1,0)--(3,4) node[right] {$H_2$};
\draw[very thick] (3,0)--(1,4) node[left] {$H_3$};

\filldraw[fill=black, draw=black] (0.5, 2.5) circle (0.1);

\filldraw[fill=black, draw=black] (1,1) circle (0.1);
\draw[thick, ->] (1,1)--++(-0.5,0.5);

\filldraw[fill=black, draw=black] (2,0.5) node[below] {${x}$} circle (0.1);
\draw[thick, ->] (2,0.5)--++(2,0.2) node[right] {${v}\in T_{x}\R^\ell$};

\filldraw[fill=black, draw=black] (2,2) circle (0.1);
\draw[thick, ->] (2,2)--++(1.5,0.4);
\draw[thick, ->] (2,2)--++(1,-1);

\filldraw[fill=black, draw=black] (2.5,3) circle (0.1);
\draw[thick, ->] (2.5,3)--++(1,0);
\draw[thick, ->] (2.5,3)--++(0,1);

%\filldraw[fill=black, draw=black] (-3,0.8) node[above] {$F^0$} circle (0.1);
%\draw[thick, dashed, ->] (-3,0.8) -- (11,0.8) node[above] {$F^1$} ;

\end{tikzpicture}
\caption{Vectors in the complexified complement.} 
\label{fig:compl}
\end{figure}

\subsection{System of half-spaces and sphere embeddings}
\label{sec:emb}

Let $\A=\{H_1, \dots, H_n\}$ be a central and essential 
hyperplane arrangement in $V=\R^\ell$. 
Let $\bm{\eps}=(\eps_1, \dots, \eps_n)$ be a 
locally consistent sign vector. 

Let $S=S^{\ell-1}=\{x\in\R^\ell\mid |x|=1\}$ 
be the unit sphere, and 
$D=D^\ell=\{x\in\R^\ell\mid |x|\leq 1\}$ 
be the unit ball in $V$. 
Recall that a vector field $\vf$ on $S$ is a continuous 
assignment of a tangent vector $\vf (x)\in T_x V$ to each 
$x\in S$. 

\begin{definition}
A vector field $\vf$ on $S$ is said to be \emph{compatible with 
a system of half-spaces $h(\bm{\eps})$} if, at each point 
$x\in S\cap H_i$ ($i=1, \dots, n$), 
$\vf(x)\neq 0$ and 
$\vf(x)$ directs towards the positive side $H_i^{\eps_i}$ of $H_i$ 
(Figure \ref{fig:vf}). 
\begin{figure}[htbp]
\centering
\begin{tikzpicture}%[scale=0.8]

%%格子
%\draw [help lines] (0,0) grid (12,6);%(0,0)から(10,4)までの"細線の方眼"
%\draw [help lines] (-4,0) grid (4,6);%(0,0)から(10,4)までの"細線の方眼"

\draw[thick] (0,0) circle (2);

\draw[thick] (-3,0)--(3,0);
\draw[thick] (240:3)--(60:3);
\draw[thick] (120:3)--(300:3);

\filldraw[fill=black, draw=black] (1,0) circle (0.05);
\draw[thick, ->] (1,0)--++(0,-0.8);

\filldraw[fill=black, draw=black] (60:0.8) circle (0.05);
\draw[thick, ->] (60:0.8)--++(-0.1,0.8);

\filldraw[fill=black, draw=black] (120:0.8) circle (0.05);
\draw[thick, ->] (120:0.8)--++(0.1,0.8);

%\filldraw[fill=black, draw=black] (0.4,0.8) circle (0.05);
%\draw[thick, ->] (0.4,0.8)--++(-0.1,0.8);

%\filldraw[fill=black, draw=black] (-0.4,0.8) circle (0.05);
%\draw[thick, ->] (-0.4,0.8)--++(0.1,0.8);

\draw[->] (0:2)--++(280:0.8);
\draw[->] (10:2)--++(300:0.9);
\draw[->] (20:2)--++(330:1);
\draw[->] (30:2)--++(0:1);
\draw[->] (40:2)--++(30:1);
\draw[->] (50:2)--++(60:1);
\draw[->] (60:2)--++(90:1);
\draw[->] (70:2)--++(90:1);
\draw[->] (80:2)--++(90:1);
\draw[->] (90:2)--++(90:1);
\draw[->] (100:2)--++(90:1);
\draw[->] (110:2)--++(90:1);
\draw[->] (120:2)--++(90:1);
\draw[->] (130:2)--++(100:0.7);
\draw[->] (140:2)--++(110:0.3);
%\draw[->] (150:2) %--++(120:1);
\filldraw[fill=black, draw=black] (150:2) circle (0.05);
\draw[->] (160:2)--++(290:0.2);
\draw[->] (170:2)--++(310:0.4);
\draw[->] (180:2)--++(330:0.6);
\draw[->] (190:2)--++(350:0.6);
\draw[->] (200:2)--++(20:0.6);
\draw[->] (210:2)--++(40:0.6);
\draw[->] (220:2)--++(60:0.6);
\draw[->] (230:2)--++(80:0.6);
\draw[->] (240:2)--++(100:0.6);
\draw[->] (250:2)--++(100:0.4);
\draw[->] (260:2)--++(100:0.2);
%\draw[->] (270:2)--++(120:1);
\filldraw[fill=black, draw=black] (270:2) circle (0.05);
\draw[->] (280:2)--++(300:0.2);
\draw[->] (290:2)--++(330:0.4);
\draw[->] (300:2)--++(0:0.6);
\draw[->] (310:2)--++(340:0.7);
\draw[->] (320:2)--++(320:0.7);
\draw[->] (330:2)--++(300:0.7);
\draw[->] (340:2)--++(290:0.7);
\draw[->] (350:2)--++(280:0.7);

\end{tikzpicture}
\caption{A vector field $\vf$ on $S$ which is 
compatible with a system of half-spaces} 
\label{fig:vf}
\end{figure}
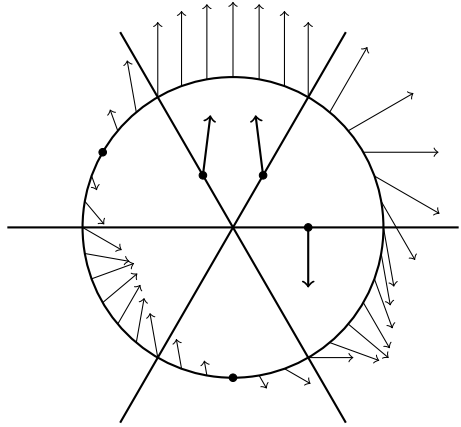
\end{definition}
Note that the local consistency of $h(\bm{\eps})$ implies 
that at each point $p\in S$, there exists a consistent 
vector field around $p$. 
Using the partition of unity, we can construct a consistent 
vector field $\theta$ on $S$. 

We construct an embedded sphere by shifting the real sphere 
$S\subset\R^\ell$ in the imaginary direction determined by the 
vector field $\vf$. 
\begin{definition}
Let $\vf$ be a vector field on $S$ compatible with $h(\bm{\eps})$. 
Define a subset $S(\bm{\eps})\subset\C^\ell$ as 
\begin{equation}
S(\bm{\eps}):=\{x+\sqrt{-1}\cdot \vf(x)\mid x\in S\}. 
\end{equation}
\end{definition}
Since the tangent vector $\vf(x)$ is not contained in $H_i$ for 
$x\in S\cap H_i$, $S(\bm{\eps})$ is an embedded sphere 
contained in the complexified complement $M(\A)$. We also 
note that the homotopy class $[S(\bm{\eps})]$ does not 
depend on the choice of the vector field $\vf$. Indeed, 
if $\vf_0$ and $\vf_1$ are vector fields on $S$ compatible with 
$h(\bm{\eps})$, then $\vf_t=(1-t)\vf_0+t\vf_1$ gives a family of 
vector fields compatible with $H(\bm{\eps})$, which provides a 
homotopy between the two embedded spheres. 

\subsection{Main result}
\label{sec:main}

We can characterize when the embedded sphere $S(\bm{\eps})$ 
is non-trivial in the homotopy group $\pi_{\ell-1}(M(\A))$. 

\begin{theorem}
\label{thm:main}
Let $\A=\{H_1, \dots, H_n\}$ be a central and essential 
hyperplane arrangement in $V=\R^\ell$. 
Let $h(\bm{\eps})$ be a 
locally consistent system of half-spaces. 
Then $S(\bm{\eps})$ is homotopically 
trivial if and only if $h(\bm{\eps})$ is globally consistent. 
\end{theorem}

The if part of Theorem \ref{thm:main} is straightforward. If 
$h(\bm{\eps})$ is globally consistent, we can extend $\vf$ to 
the vector field $\widetilde{\vf}$ over the entire space 
such that $\vf(x)\in H_i^{\eps_i}$ whenever $x\in H_i$. 
Then we can fill the sphere $S(\bm{\eps})$ with the embedded disk 
\[
D(\bm{\eps})=\{x+\sqrt{-1}\cdot\widetilde{\vf}(x)\mid x\in D\}. 
\]
(Figure \ref{fig:ifpart}). Note that the disk is a cell in the 
Salvetti complex \cite{sal-top}. 
\begin{figure}[htbp]
\centering
\begin{tikzpicture}%[scale=0.8]

%%格子
%\draw [help lines] (0,0) grid (12,6);%(0,0)から(10,4)までの"細線の方眼"
%\draw [help lines] (-4,0) grid (4,6);%(0,0)から(10,4)までの"細線の方眼"

\draw[thick] (0,0) circle (1.5);

\draw[thick] (-2.5,0)--(2.5,0);
\draw[thick] (240:2.5)--(60:2.5);
\draw[thick] (120:2.5)--(300:2.5);

\filldraw[fill=black, draw=black] (2,0) circle (0.05);
\draw[thick, ->] (2,0)--++(0,0.8);

\filldraw[fill=black, draw=black] (60:2) circle (0.05);
\draw[thick, ->] (60:2)--++(0,0.8);

\filldraw[fill=black, draw=black] (120:2) circle (0.05);
\draw[thick, ->] (120:2)--++(0,0.8);

\draw[->] (0:1.5)--++(90:0.4);
\draw[->] (10:1.5)--++(90:0.4);
\draw[->] (20:1.5)--++(90:0.4);
\draw[->] (30:1.5)--++(90:0.4);
\draw[->] (40:1.5)--++(90:0.4);
\draw[->] (50:1.5)--++(90:0.4);
\draw[->] (60:1.5)--++(90:0.4);
\draw[->] (70:1.5)--++(90:0.4);
\draw[->] (80:1.5)--++(90:0.4);
\draw[->] (90:1.5)--++(90:0.4);
\draw[->] (100:1.5)--++(90:0.4);
\draw[->] (110:1.5)--++(90:0.4);
\draw[->] (120:1.5)--++(90:0.4);
\draw[->] (130:1.5)--++(90:0.4);
\draw[->] (140:1.5)--++(90:0.4);
\draw[->] (150:1.5)--++(90:0.4);
\draw[->] (160:1.5)--++(90:0.4);
\draw[->] (170:1.5)--++(90:0.4);
\draw[->] (180:1.5)--++(90:0.4);
\draw[->] (190:1.5)--++(90:0.4);
\draw[->] (200:1.5)--++(90:0.4);
\draw[->] (210:1.5)--++(90:0.4);
\draw[->] (220:1.5)--++(90:0.4);
\draw[->] (230:1.5)--++(90:0.4);
\draw[->] (240:1.5)--++(90:0.4);
\draw[->] (250:1.5)--++(90:0.4);
\draw[->] (260:1.5)--++(90:0.4);
\draw[->] (270:1.5)--++(90:0.4);
\draw[->] (280:1.5)--++(90:0.4);
\draw[->] (290:1.5)--++(90:0.4);
\draw[->] (300:1.5)--++(90:0.4);
\draw[->] (310:1.5)--++(90:0.4);
\draw[->] (320:1.5)--++(90:0.4);
\draw[->] (330:1.5)--++(90:0.4);
\draw[->] (340:1.5)--++(90:0.4);
\draw[->] (350:1.5)--++(90:0.4);

\draw[->] (0:1)--++(90:0.4);
\draw[->] (30:1)--++(90:0.4);
\draw[->] (60:1)--++(90:0.4);
\draw[->] (90:1)--++(90:0.4);
\draw[->] (120:1)--++(90:0.4);
\draw[->] (150:1)--++(90:0.4);
\draw[->] (180:1)--++(90:0.4);
\draw[->] (210:1)--++(90:0.4);
\draw[->] (240:1)--++(90:0.4);
\draw[->] (270:1)--++(90:0.4);
\draw[->] (300:1)--++(90:0.4);
\draw[->] (330:1)--++(90:0.4);

\draw[->] (0:0.5)--++(90:0.4);
\draw[->] (60:0.5)--++(90:0.4);
\draw[->] (120:0.5)--++(90:0.4);
\draw[->] (180:0.5)--++(90:0.4);
\draw[->] (240:0.5)--++(90:0.4);
\draw[->] (300:0.5)--++(90:0.4);

\draw[->] (0:0)--++(90:0.4);

\end{tikzpicture}
\caption{Filling $S(\bm{\eps})$ by a disk.} 
\label{fig:ifpart}
\end{figure}
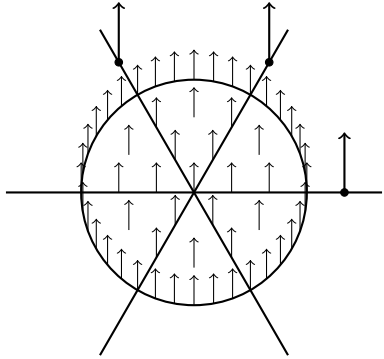
We will prove the only if part in the sequel. 

\section{Proofs}
\label{sec:proof}

\subsection{Twisted intersection numbers and the Hurewicz map}
\label{sec:twist}

In this section we recall basic facts about local system 
homology groups, particularly the twisted intersection numbers 
developed in \cite{ao-kit, kit-yos}. Let $M$ be a connected 
oriented $C^\infty$-manifold with $\dim_{\R}M=d$. 
We also assume that $M$ does not have the boundary. 
Let $\scL$ be a local system of $\K$-vector spaces on $M$. 
Denote the fiber at $x\in M$ by $\scL_x$. Then for each 
curve $\gamma:[0, 1]\longrightarrow M$, we have the 
isomorphism 
$\rho(\gamma):\scL_{\gamma(0)}\stackrel{\simeq}{\longrightarrow}
\scL_{\gamma(1)}$, which is called the parallel transport along 
$\gamma$. Note that the isomorphism $\rho(\gamma)$ 
depends only on the homotopy class of $\gamma$. 

Let $\scL^{\vee}$ is the dual local system. 
We describe explicitly the twisted intersection pairing 
\begin{equation}
\mathcal{I}: H_{d-k}(M, \scL^\vee)\otimes_{\K} H_k^{\BM}(M, \scL)
\longrightarrow\K
\end{equation}
for special cases. 
Let $W$ be a $k$-dimensional oriented closed submanifold of 
$M$ without boundary. 
Let $i:W\hookrightarrow M$ be the inclusion. 

Let $Z$ be a $(d-k)$-dimensional oriented compact manifold 
without boundary. Let $f:Z\longrightarrow M$ be a 
$C^\infty$-map that is transversal to $W$. 
For each $p\in f^{-1}(W)$, the local intersection number 
$I_p(f(Z), W)\in\{\pm 1\}$ is determined as 
\begin{equation}
I_p(f(Z), W)=
\begin{cases}
+1, & \mbox{ if }f_*T_pZ\oplus T_{f(p)}W \mbox{ is positively oriented}\\
-1, & \mbox{ if }f_*T_pZ\oplus T_{f(p)}W \mbox{ is negatively oriented.}
\end{cases}
\end{equation}
Let $\sigma\in H^0(W, i^*\scL)$. Then the fundamental 
cycle $[W]$ and $\sigma$ determine a twisted Borel-Moore 
cycle 
$i_*([W]\otimes\sigma)\in H_k^{\BM}(M, \scL)$. 
Similarly, $\tau\in H^0(Z, f^*\scL^\vee)$ determines 
$f_*([Z]\otimes\tau)\in H_{d-k}(M, \scL^\vee)$. 
For these two cycles, the twisted intersection number is described 
as follows. 
\begin{equation}
\label{eq:int}
\mathcal{I}(f_*([Z]\otimes\tau), i_*([W]\otimes\sigma))=
\sum_{p\in f^{-1}(W)}I_p(f(Z), W)\langle\tau(f(p)), \sigma(f(p))\rangle, 
\end{equation}
where $\langle -, -\rangle$ is the natural pairing 
between $\scL_{f(p)}^\vee$ and 
$\scL_{f(p)}$. 

Next let us recall the twisted Hurewicz map. Let $x_0\in M$ be 
the base point. If $k\geq 2$, then $S^k$ is simply connected. 
Therefore, for any continuous map $f:(S^k, *)\to (M, x_0)$, 
the pull-back $f^*\scL$ is a trivial local system. 
In particular, we have $H^0(S^k, f^*\scL)\simeq \scL_{x_0}$. 
Thus an element of $\scL_{x_0}$ gives a global section, and 
we have the following twisted Hurewicz map \cite{yos-twi}
\begin{equation}
\eta_{\scL, x_0}: \pi_{k}(M(\A), x_0)\otimes_\Z\scL_{x_0}
\longrightarrow
H_{k}(M(\A), \scL). 
\end{equation}

\subsection{Non-triviality of the sphere}
\label{sec:nontriv}

Now we return to the proof of 
Theorem \ref{thm:main}. 
Suppose that the system of half-spaces $h(\bm{\eps})$ 
is locally consistent, but globally inconsistent. 
We shall prove that $[S(\bm{\eps})]$ is 
not homotopically trivial. 

By Proposition \ref{prop:sink}, there exists a sink $C$. 
By definition, $C$ is contained in $H_i^{\eps_i}$ for all walls 
$H_i$ of $C$. By the assumption, there exists a hyperplane, 
say $H_n$, 
such that $C\not\subset H_n^{\eps_n}$. 
(Otherwise, $h(\bm{\eps})$ is globally consistent.) 

Let $x_0\in S\cap (-C)$. We can construct a vector field 
$\vf$ on $S$ satisfying the following properties (Figure \ref{fig:euler}). 
\begin{itemize}
\item[(i)] 
$\vf(x)\in T_xS$ for all $x\in S$. 
\item[(ii)] 
$\vf(\pm x_0)=0$, and $\pm x_0$ are the only zeros of $\vf$ on 
$(\mp C)\cap S$. 
\item[(iii)] 
Let $(z_1, \dots, z_{\ell-1})$ be a local coordinates of $S$ around 
$x_0\in S$, and the vector field $\vf$ is equal to 
the Euler vector field $\sum_{i=1}^{\ell-1}z_i\frac{\partial}{\partial z_i}$ 
around $x_0$. 
\item[(iv)] 
Let $(z_1, \dots, z_{\ell-1})$ be a local coordinates of $S$ around 
$-x_0\in S$, and the vector field $\vf$ is equal to the 
negative Euler vector field 
$-\sum_{i=1}^{\ell-1}z_i\frac{\partial}{\partial z_i}$ 
around $-x_0$. 
\end{itemize}
\begin{figure}[htbp]
\centering
\begin{tikzpicture}%[scale=0.8]

%%格子
%\draw [help lines] (0,0) grid (12,6);%(0,0)から(10,4)までの"細線の方眼"
%\draw [help lines] (-4,0) grid (4,6);%(0,0)から(10,4)までの"細線の方眼"

\draw[thin] (0,0) circle (2.97);

\draw[thick] (-3.5,0)--(3.5,0) node[right] {$H_n$};
\draw[thick] (240:3.5)--(60:3.5);
\draw[thick] (120:3.5)--(300:3.5);

\filldraw[fill=black, draw=black] (2,0) circle (0.05);
\draw[thick, ->] (2,0)--++(0,-0.8);
\filldraw[fill=black, draw=black] (-2,0) circle (0.05);
\draw[thick, ->] (-2,0)--++(0,-0.8);

\filldraw[fill=black, draw=black] (60:2) circle (0.05);
\draw[thick, ->] (60:2)--++(0,0.8);
\filldraw[fill=black, draw=black] (240:2) circle (0.05);
\draw[thick, ->] (240:2)--++(0,0.8);

\filldraw[fill=black, draw=black] (120:2) circle (0.05);
\draw[thick, ->] (120:2)--++(0,0.8);
\filldraw[fill=black, draw=black] (300:2) circle (0.05);
\draw[thick, ->] (300:2)--++(0,0.8);

\draw (0,2) node {$C$}; 
\draw (0,-2) node {$-C$}; 
\draw (20:3) node[left] {$S$}; 
\filldraw[fill=black, draw=black] (0,2.97) node [above] {$-x_0$} circle (0.05);
\filldraw[fill=black, draw=black] (0,-2.97) node [above] {$x_0$} circle (0.05);

\filldraw[fill=black, draw=black] (30:2.97) circle (0.05);
\filldraw[fill=black, draw=black] (150:2.97) circle (0.05);
\filldraw[fill=black, draw=black] (210:2.97) circle (0.05);
\filldraw[fill=black, draw=black] (330:2.97) circle (0.05);

\draw[->, >=stealth] (0:3)--++(275:0.8);
\draw[->, >=stealth] (10:3)--++(285:0.6);
\draw[->, >=stealth] (20:3)--++(295:0.3);
\draw[->, >=stealth] (26:3)--++(300:0.2);

\draw[->, >=stealth] (34:3)--++(119:0.2);
\draw[->, >=stealth] (40:3)--++(125:0.3);
\draw[->, >=stealth] (50:3)--++(135:0.6);
\draw[->, >=stealth] (60:3)--++(145:0.6);
\draw[->, >=stealth] (70:3)--++(155:0.6);
\draw[->, >=stealth] (80:3)--++(165:0.3);
\draw[->, >=stealth] (85:3)--++(171:0.2);

\draw[->, >=stealth] (95:3)--++(9:0.2);
\draw[->, >=stealth] (100:3)--++(15:0.3);
\draw[->, >=stealth] (110:3)--++(25:0.6);
\draw[->, >=stealth] (120:3)--++(35:0.6);
\draw[->, >=stealth] (130:3)--++(45:0.6);
\draw[->, >=stealth] (140:3)--++(55:0.3);
\draw[->, >=stealth] (146:3)--++(61:0.2);

\draw[->, >=stealth] (154:3)--++(240:0.2);
\draw[->, >=stealth] (160:3)--++(245:0.3);
\draw[->, >=stealth] (170:3)--++(255:0.6);
\draw[->, >=stealth] (180:3)--++(265:0.8);

\draw[->, >=stealth] (190:3)--++(275:0.6);
\draw[->, >=stealth] (200:3)--++(285:0.3);
\draw[->, >=stealth] (205:3)--++(291:0.2);

\draw[->, >=stealth] (215:3)--++(129:0.2);
\draw[->, >=stealth] (221:3)--++(135:0.3);
\draw[->, >=stealth] (231:3)--++(145:0.6);
\draw[->, >=stealth] (241:3)--++(155:0.6);
\draw[->, >=stealth] (251:3)--++(165:0.6);
\draw[->, >=stealth] (261:3)--++(175:0.3);
\draw[->, >=stealth] (266:3)--++(181:0.2);

\draw[->, >=stealth] (274:3)--++(359:0.2);
\draw[->, >=stealth] (279:3)--++(5:0.3);
\draw[->, >=stealth] (289:3)--++(15:0.6);
\draw[->, >=stealth] (299:3)--++(25:0.6);
\draw[->, >=stealth] (309:3)--++(35:0.6);
\draw[->, >=stealth] (319:3)--++(45:0.3);
\draw[->, >=stealth] (325:3)--++(51:0.2);
\draw[->, >=stealth] (335:3)--++(247:0.2);
\draw[->, >=stealth] (340:3)--++(255:0.3);
\draw[->, >=stealth] (350:3)--++(265:0.6);

\draw[->, >=stealth] (0:3)--++(275:0.8);
\draw[->, >=stealth] (0:3)--++(275:0.8);
\draw[->, >=stealth] (0:3)--++(275:0.8);
\draw[->, >=stealth] (0:3)--++(275:0.8);

\end{tikzpicture}
\caption{A vector field tangent to $S$} 
\label{fig:euler}
\end{figure}
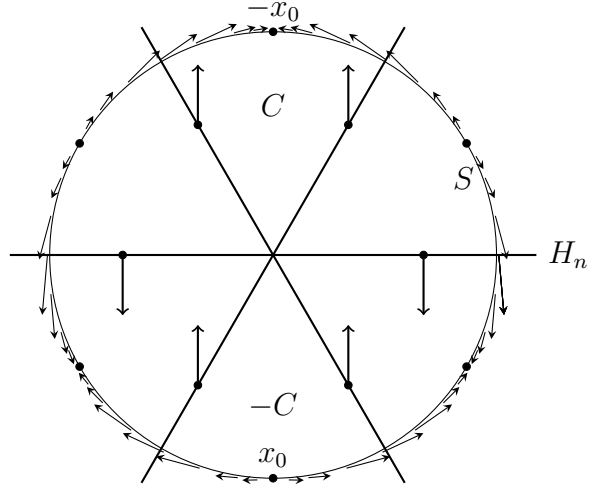

Now we define a $\C$-local system $\scL$ of rank one and 
a Borel-Moore cycle of degree $(\ell+1)$. Such a 
local system is determined by a homomorphism 
$H_1(M(\A), \Z)\to\C^\times$. Since $H_1(M(\A, \Z)$ is freely 
generated by meridians of each hyperplane $H_i$, the local system 
$\scL$ is determined by the local monodromies $q_i\in\C^\times$ 
around $H_i$ ($i=1, \dots, n$). 

Without loss of generality, we may assume that 
$C\subset H_i^{\eps_i}$ for $i=1, \dots, k$ and 
$C\not\subset H_i^{\eps_i}$ for $i=k+1, \dots, n$ 
(note that $1<k<n$). 
Let $(q_1, \dots, q_n)\in(\C^\times)^n$. Assume that 
$(q_1, \dots, q_n)$ satisfies the following conditions: 
\begin{equation}
\label{eq:condition}
\prod_{i=1}^nq_i=1, \ \ \ \mbox{ and } \ \ \ 
\prod_{i=k+1}^nq_i\neq 1.
\end{equation}
Such a local system exists, for example, if we take 
$q_i=e^{2\pi\sqrt{-1}/n}$ for all $i$, the condition is satisfied.   
Let $\scL$ be a rank one local system defined by 
such $(q_1, \dots, q_n)$. 

Now, we define a closed submanifold $W\subset M(\A)$ as follows: 
\[
W=\{x\cos\theta+\sqrt{-1}\cdot x\sin\theta\mid
x\in C, 0\leq\theta\leq 2\pi\}. 
\]
In other words, 
$x+\sqrt{-1}\cdot v$ in contained in $W$ if and only if 
\begin{itemize}
\item
$x=0$ with $v\in C\cup(-C)$, or 
\item 
%$x\neq 0$ with 
$x\in C\cup(-C)$ with 
$v\in\R\cdot x$. 
\end{itemize}
(See Figure \ref{fig:W}.) 

\begin{figure}[htbp]
\centering
\begin{tikzpicture}%[scale=0.8]

%%格子
%\draw [help lines] (0,0) grid (12,6);%(0,0)から(10,4)までの"細線の方眼"
%\draw [help lines] (-4,0) grid (4,6);%(0,0)から(10,4)までの"細線の方眼"

%\draw[thick] (0,0) circle (1.5);

\draw[thick] (-2.5,0)--(2.5,0);
\draw[thick] (230:2.5)--(50:2.5);
\draw[thick] (130:2.5)--(310:2.5);

\draw (1,1.5) node[above] {$C$}; 

\draw[thin, dashed] (260:2.1)--(80:2.1);
\draw[thin, dashed] (115:2.2)--(295:2.2);

\filldraw[fill=black, draw=black] (0,0) circle (0.04);
\draw[thick, ->] (0,0)--++(80:0.7);
\draw[thick, ->] (0,0)--++(260:0.7);
\filldraw[fill=black, draw=black] (260:1.5) circle (0.04);
\draw[thick, ->] (260:01.5)--++(80:0.5);
\draw[thick, ->] (260:01.5)--++(260:0.5);

\filldraw[fill=black, draw=black] (80:0.9) circle (0.04);
\draw[thick, ->] (80:0.9)--++(80:0.7);

\filldraw[fill=black, draw=black] (115:1.3) circle (0.04);
\draw[thick, ->] (115:1.3)--++(115:0.7);
\draw[thick, ->] (115:1.3)--++(295:0.5);

\filldraw[fill=black, draw=black] (295:1) circle (0.04);
\draw[thick, ->] (295:1)--++(115:0.6);
\draw[thick, ->] (295:1)--++(295:0.8);

\end{tikzpicture}
\caption{The submanifold $W$} 
\label{fig:W}
\end{figure}
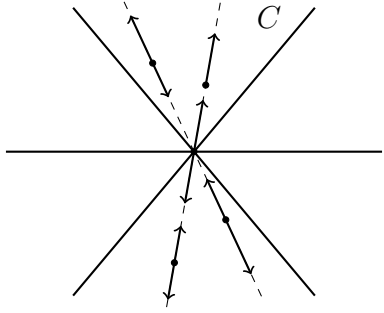

\begin{lemma}
\label{lem:triv}
The restrictions $\scL|_W$ and $\scL|_{S(\bm{\eps})}$ are trivial. 
\end{lemma}
\begin{proof}[Proof of Lemma \ref{lem:triv}]
Since $W$ is homotopic to the circle $S^1$ which turns 
around all hyperplanes $H_1, \dots, H_n$, and by the 
first relation in (\ref{eq:condition}), 
$\scL$ does not have monodromy along this $S^1$, 
hence $\scL$ is trivial on $W$. 

If $\ell\geq 3$, then $S(\bm{\eps})\simeq S^{\ell-1}$ is simply 
connected. Therefore, the restriction on $S(\bm{\eps})$ is trivial. 
When $\ell =2$, the $S(\bm{\eps})\simeq S^1$ is homologically 
trivial. Hence $\scL$ does not have monodromy along 
$S(\bm{\eps})\simeq S^1$, 
hence $\scL$ is trivial on $S(\bm{\eps})$. 
\end{proof}

\begin{proof}[Proof of Theorem \ref{thm:main}]
By the constructions, $W\cap S(\bm{\eps})=\{\pm x_0\}$. 
They intersect transversally, and 
we can choose orientations of $S(\bm{\eps})$ and $W$ 
in such a way that $I_{x_0}(S(\bm{\eps}), W)=1$ and 
$I_{-x_0}(S(\bm{\eps}), W)=-1$. 
Let $e\in \scL_{x_0}$ and $e^*\in\scL_{x_0}^\vee$ be bases 
such that $\langle e, e^*\rangle=1$. 

Let $\gamma_1: [0, \pi]\to W$ be the curve defined by 
\[
\gamma_1(\theta)=\cos\theta\cdot x_0-\sqrt{-1}\sin\theta\cdot x_0. 
\]
Then $\gamma_1$ is a curve on $W$ from $x_0$ to $-x_0$. 

Let $\gamma_2:[0, pi]\to S(\bm{\eps})$ be the curve on 
$S(\bm{\eps})$ from $x_0$ to $-x_0$. Note that if $\ell\geq 3$, 
$\gamma_2$ is unique up to homotopy. When $\ell=2$, we choose 
$\gamma_2$ to be one of the half-circles (Figure \ref{fig:curves}). 
\begin{figure}[htbp]
\centering
\begin{tikzpicture}%[scale=0.8]

%%格子
%\draw [help lines] (0,0) grid (12,6);%(0,0)から(10,4)までの"細線の方眼"
%\draw [help lines] (-4,0) grid (4,6);%(0,0)から(10,4)までの"細線の方眼"

\filldraw[fill=black, draw=black] (0,-1.5) node [below] {$x_0$} circle (0.05);
\filldraw[fill=black, draw=black] (0,1.5) node [above] {$-x_0$} circle (0.05);

\draw[dashed] (0,0) circle (1.5);

\draw[thick] (-2.5,0)--(2.5,0) node[right] {$H_n$};
\draw[thick] (220:2.5)--(40:2.5);
\draw[thick] (120:2.5)--(300:2.5);

\filldraw[fill=black, draw=black] (-2,0) circle (0.05);
\draw[thick, ->] (-2,0)--++(0,-0.8);

\filldraw[fill=black, draw=black] (40:2) circle (0.05);
\draw[thick, ->] (40:2)--++(0,0.8);

\filldraw[fill=black, draw=black] (120:2) circle (0.05);
\draw[thick, ->] (120:2)--++(0,0.8);

\draw[thick, ->] (0, -1.5)--++(0,0.4);
\draw[thick, ->] (0, -1)--++(0,0.4);
\draw (0, -0.7) node[left] {$\gamma_1$};
\draw[thick, ->] (0, -0.5)--++(0,0.4);
\draw[thick, ->] (0, 0)--++(0,0.4);
\draw[thick, ->] (0, 0.5)--++(0,0.4);
\draw[thick, ->] (0, 1)  --++(0,0.4);

\draw[thick] (0, -1.5) arc [radius=1.5, start angle = -90, end angle=90];

\draw[thick, ->] (300:1.5) -- ++(25:0.6);
\draw[thick, ->] (0:1.5) -- ++(275:0.6) node[right] {$\gamma_2$};
\draw[thick, ->] (40:1.5) -- ++(125:0.6);

\end{tikzpicture}
\caption{Curves $\gamma_1$ and $\gamma_2$.} 
\label{fig:curves}
\end{figure}
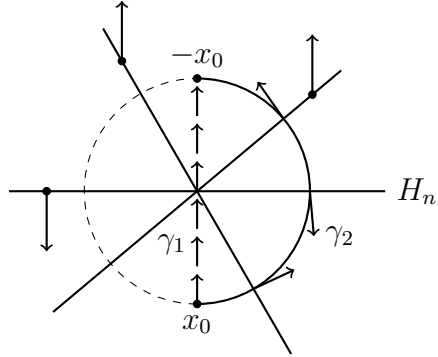

Consider the closed curve $\gamma_1^{-1}\cdot\gamma_2$. This 
curve turns each of $H_{k+1}, \dots, H_n$ positively. Hence the 
monodromy of $\scL$ along $\gamma_1^{-1}\cdot\gamma_2$ is 
$\lambda::=\prod_{i=k+1}^n q_i\neq 1$. 

Now let us compute the twisted intersection number 
$\mathcal{I}([W]\otimes e, [S(\bm{\eps})]\otimes e^*)$. 
By the formula (\ref{eq:int}), we have 
\[
\begin{split}
\mathcal{I}([W]\otimes e, [S(\bm{\eps})]\otimes e^*)&=
\langle e, e^*\rangle -
\langle \rho(\gamma_1)e, \rho(\gamma_2)e^*\rangle\\
&=
\langle e, e^*\rangle -
\langle \rho(\gamma_1)e, \lambda \rho(\gamma_1)e^*\rangle\\
&=1-\lambda\neq 0. 
\end{split}
\]
Hence 
$[S(\bm{\eps})]\otimes e^*\in H_{\ell-1}(M(\A), \scL^\vee)$ and 
$[W]\otimes e\in H_{\ell+1}^{\BM}(M(\A), \scL)$ 
are both non-trivial, and 
$S(\bm{\eps})$ is not homotopically trivial. 
\end{proof}

\section{Non-$K(\pi, 1)$ arrangements}
\label{sec:cor}

\subsection{An obstruction to $K(\pi, 1)$}
\label{sec:obst}

Let $\A=\{H_1, \dots, H_n\}$ be a central essential arrangement 
in $V=\R^\ell$. 

For $1\leq k\leq \ell$ let us define $\Sigma_k$ as 
\[
\Sigma_k:=\{h(\bm{\eps})\mid h(\bm{\eps})\mbox{ is consistent 
at any $X\in L(\A)$ with }\codim X\leq k\}. 
\]
Note that $\Sigma_1$ is the set of all 
systems of half-spaces and $\Sigma_\ell$ is the set of all globally 
consistent system of half-spaces. So, $\#\Sigma_1=2^n$ and 
$\#\Sigma_\ell$ is equal to the number of chambers. 
We have a natural filtration 
\begin{equation}
\Sigma_1\supset \Sigma_2\supset \cdots \supset \Sigma_{\ell-1}
\supset \Sigma_\ell. 
\end{equation}
A system $h(\bm{\eps})$ is 
consistent (resp. locally consistent) 
if and only if $h(\bm{\eps})\in\Sigma_{\ell}$. 
(resp. $h(\bm{\eps})\in\Sigma_{\ell-1}$). 
The cardinalities of these sets reflect the structure of 
the homotopy groups. 

\begin{theorem}
\label{thm:obst}
\begin{itemize}
\item[(1)] 
Suppose that there exists $k\geq2$ satisfying 
\begin{equation}
\label{eq:gap}
%\Sigma_2=\dots =
\Sigma_k\supsetneq\Sigma_{k+1}. 
\end{equation}
Then $\pi_k(M(\A))\neq 0$. 
\item[(2)] 
If $M(\A)$ is $K(\pi, 1)$ space, then 
\begin{equation}
\label{eq:equal}
\Sigma_2=\Sigma_3=\dots =\Sigma_\ell. 
\end{equation}
\end{itemize}
\end{theorem}
\begin{proof}
(1) There exist a system of half-spaces $h(\bm{\eps})$ and 
an intersection $X\in L(\A)$ with $\codim X=k+1$ 
such that $h(\bm{\eps})$ is inconsistent at $X$ but 
consistent for any $Y\supsetneq X$. Theorem \ref{thm:main} 
enables us to construct a $k$-dimensional sphere that is not 
homotopically trivial. 

(2) is easily obtained from (1). 
\end{proof}

\begin{remark}
The condition (\ref{eq:equal}) 
$\Sigma_2=\cdots=\Sigma_\ell$ is closely related to 
the notion of biclosed subset of vectors and clean arrangements 
in \cite{mcc-bic, bar-spe}. It was proved in \cite{mcc-bic} 
that supersolvable arrangements and simplicial arrangements 
satisfy the condition (\ref{eq:equal}). 
Since supersolvable and simplicial arrangements are known to be 
$K(\pi, 1)$, Theorem \ref{thm:obst} (2) generalizes the result. 
\end{remark}

%\subsection{Non-$K(\pi, 1)$ arrangements}
%\label{sec:non}

\subsection{Examples}
\label{sec:ex}

\begin{proposition}
\label{prop:generic}
Let $\A=\{H_1, \dots, H_n\}$ be a central essential arrangement 
in $V=\R^\ell$. Let $\mathcal{B}=\{H_{n+1}, \dots, H_{n+k}\}$ be 
an any central arrangement. Then for generic $g\in GL(V)$, 
$\A\cup g\mathcal{B}=\{H_1, \dots, H_n, g(H_{n+1}), \dots, 
g(H_{n+k})$\} is non-$K(\pi, 1)$. 
\end{proposition}

\begin{proof}
We say that two linear subspaces $X_1, X_2\subset V$ are in 
general position if either 
\begin{itemize}
\item
$\dim X_1+\dim X_2<\ell$ and $X_1\cap X_2=\{0\}$, or 
\item 
$\dim X_1+\dim X_2\geq \ell$ and $\dim X_1\cap X_2=
\dim X_1+\dim X_2-\ell$.  
\end{itemize}
Note that there exists a non-empty Zariski open subset $U$ 
of $GL(V)$ such that for any $X\in L(\A)$ and 
$Y\in L(g\mathcal{B})$, $X$ and $Y$ are in general position. 
The topological type of the complexified complement 
$M(\A\cup g\mathcal{B})$ is stable for generic $g\in GL(V)$ 
\cite{ran-iso, naz-yos}. 

Based on Theorem \ref{thm:main}, it is sufficient to construct 
a system of half-spaces that is locally consistent but 
globally inconsistent. 
We first choose a chamber $C_0$ of $\A$ and 
choose half-spaces such that 
$C_0=H_1^{\eps_1}\cap \dots, \cap H_n^{\eps_n}$. 
Now take $g\in GL(V)$ in such a way that $gH_{n+1}$ does not 
separate $C_0$ and choose a half-space satisfying 
$C_0\not\subset gH_{n+1}^{\eps_{n+1}}$. Choose a chamber 
$C_1$ of $\mathcal{B}$ and half-spaces satisfying 
$C_1=gH_{n+1}^{\eps_{n+1}}\cap\cdots\cap gH_{n+k}^{\eps_{n+k}}$. 
By the transversality of $\A$ and $g\mathcal{B}$, the system of 
half-spaces 
$(H_i^{\eps_i}\mid i=1, \dots, n+k)$ 
is locally consistent but globally inconsistent. 
\end{proof}

Proposition \ref{prop:generic} recovers known non-$K(\pi, 1)$ 
arrangements. 

\begin{example}
The following arrangements are not $K(\pi, 1)$. 
\begin{itemize}
\item
(Hattori \cite{hat-top}) Let $\A=\{H_1, \dots, H_n\}$ be a generic 
central arrangement in $\R^\ell$ with $n>\ell$. 
\item 
(Papadima, Suciu \cite{ps-h}, Yoshinaga \cite{yos-twi}) 
Let $\A$ be a central essential arrangement in $V=\R^\ell$. 
Let $F\subset V$ be a generic (linear) hyperplane. Then 
$\A\cup\{F\}$ is not $K(\pi, 1)$. 
\end{itemize}
\end{example}

As we saw in Theorem \ref{thm:obst}, 
$K(\pi, 1)$-property implies 
the condition (\ref{eq:equal}) $\Sigma_2=\Sigma_\ell$. 
However, the converse does not hold in general. 
\begin{example}
\label{ex:falk}
(\cite[\S 3.4]{fr-hom2}) 
Let $X_2=\{H_1, \dots, H_6\}$ 
be the affine line arrangement as in Figure \ref{fig:falk}. 
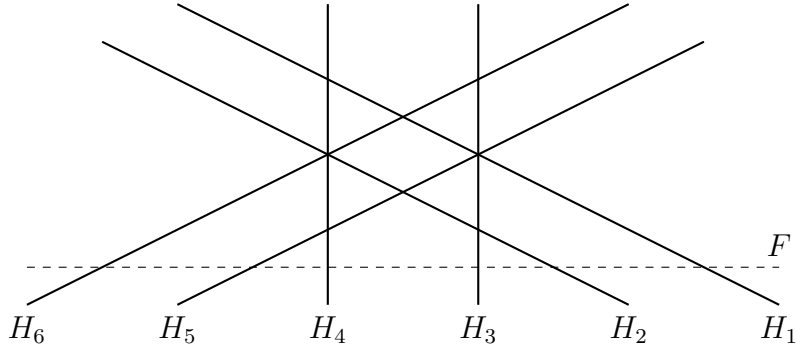
\begin{figure}[htbp]
\centering
\begin{tikzpicture}%[scale=0.8]

%%格子
%\draw [help lines] (0,0) grid (12,6);%(0,0)から(10,4)までの"細線の方眼"
%\draw [help lines] (-4,0) grid (4,6);%(0,0)から(10,4)までの"細線の方眼"

\draw[thick] (0,0) node[below] {$H_6$} -- ++(8,4);
\draw[thick] (2,0) node[below] {$H_5$} -- ++(7,3.5);
\draw[thick] (4,0) node[below] {$H_4$} -- ++(0,4);
\draw[thick] (6,0) node[below] {$H_3$} -- ++(0,4);
\draw[thick] (8,0) node[below] {$H_2$} -- ++(-7,3.5);
\draw[thick] (10,0) node[below] {$H_1$} -- ++(-8,4);

\draw[dashed] (0,0.5)--(10,0.5) node[above] {$F$}; 

\end{tikzpicture}
\caption{The arrangement $X_2$} 
\label{fig:falk}
\end{figure}
Let $\gamma_i$ ($i=1, \dots, 6$) be the meridian cycle of $H_i$ 
on the complexified generic line $F\otimes\C$. Then the fundamental group 
$\pi_1(M(X_2))$ is generated by $\gamma_1, \dots, \gamma_6$ 
with the following relations. 
\[
\begin{split}
&
[\gamma_2, \gamma_3]=1, 
[\gamma_4, \gamma_5]=1, 
[\gamma_2, \gamma_5]=1, 
[\gamma_1, \gamma_6]=1, 
[\gamma_3, \gamma_6]=1, 
[\gamma_1, \gamma_4]=1\\
&
\gamma_1\gamma_3\gamma_5=
\gamma_3\gamma_5\gamma_1=
\gamma_5\gamma_1\gamma_3, \ \ \ 
\gamma_2\gamma_4\gamma_6=
\gamma_4\gamma_6\gamma_2=
\gamma_6\gamma_2\gamma_4.  
\end{split}
\]
From these relations, we can prove that the three elements 
$\gamma_2, [\gamma_3, \gamma_4]$ and $\gamma_5$ are 
mutually commutative. Hence the fundamental group 
$\pi_1(M(X_2))$ contains a subgroup $G$ which is isomorphic to 
the free abelian group $\Z^3$. This implies that the 
complement $M=M(X_2)$ is not a $K(\pi, 1)$-space 
(unpublished work by L. Paris \cite{fr-hom2}) . 
Indeed, if $M$ is $K(\pi, 1)$, there exists a free action 
of $G=\Z^3$ on the universal covering $\widetilde{M}$. 
Note that $M$ is homotopy equivalent to a $2$-dimensional 
CW-complex. Hence $\widetilde{M}/G$, the $K(G, 1)$ space, 
is homotopic to $2$-dimensional CW complex. In particular, 
$H_3(\widetilde{M}/G, \Z)=0$. 
However, this 
contradicts the fact that $K(\Z^3, 1)$ space has homotopy type 
of the $3$-torus $(S^1)^3$, which has a non-zero third homology group. 

Let $\A=cX_2$ be the coning of $X_2$. 
Using Proposition \ref{prop:sink}, we can check that $cX_2$ 
satisfies the condition (\ref{eq:equal}) $\Sigma_2=\Sigma_3$. 
Thus the condition (\ref{eq:equal}) does not imply 
$K(\pi, 1)$-property. 
\end{example}

\begin{question}
Can one detect the non-triviality of the homotopy group 
$\pi_k(M(cX_2))$ by using the twisted Hurewicz maps? 
\end{question}

\medskip

\noindent
\emph{Acknowledgements.} 
The construction of the sphere in this paper was first 
presented in the conference 
``Hot Topics: Artin Groups and Arrangements - 
Topology, Geometry, and Combinatorics'' at SLMath, Berkeley, 
March 2024. 
The author deeply appreciates many comments by the audience. 
In particular, he thanks Mike Falk on Example \ref{ex:falk} and 
Grant Barkley for references. 
This work was partially supported by JSPS KAKENHI 
Grant Numbers JP18H01115, JP23H00081, and JP21H00975.

\end{document}